\documentclass[letterpaper,12pt]{article}

\title{Matroid lifts and representability}

\author{Daniel Irving Bernstein and Zach Walsh}
\date{December 20, 2023}

\usepackage{latexsym,array,delarray,amsthm,amssymb,epsfig,mathtools,enumerate,framed}
\usepackage[numbers]{natbib}
\usepackage{tikz,url}
\tikzstyle{vertex}=[circle, draw, inner sep=0pt,minimum size=6pt, fill=black]

\usetikzlibrary{calc}
\usepackage{verbatim}
\usepackage{graphicx}
\usepackage{centernot}
\usepackage{hyperref}

\theoremstyle{plain}
\newtheorem{thm}{Theorem}
\newtheorem{claim}{Claim}[thm]

\newtheorem{prop}[thm]{Proposition}
\newtheorem{cor}[thm]{Corollary}
\newtheorem{conj}[thm]{Conjecture}
\newtheorem*{thm*}{Theorem}
\newtheorem*{lemma*}{Lemma}
\newtheorem*{prop*}{Proposition}
\newtheorem*{cor*}{Corollary}
\newtheorem*{conj*}{Conjecture}

\theoremstyle{definition}
\newtheorem{defn}[thm]{Definition}
\newtheorem*{defn*}{Definition}

\newtheorem{ques}[thm]{Question}

\newcommand{\cM}{\mathcal{M}}
\newcommand{\cC}{\mathcal{C}}
\newcommand{\cH}{\mathcal{H}}
\newcommand{\cA}{\mathcal{A}}

\newcommand{\bZ}{\mathbb{Z}}
\newcommand{\ep}{\epsilon}

\DeclareMathOperator{\cl}{cl}
\newcommand{\del}{\!\setminus\!}

\begin{document}

\maketitle

\begin{abstract}
    A 1965 result of Crapo shows that every elementary lift of a matroid $M$ can be constructed from a linear class of circuits of $M$.
    In a recent paper, Walsh generalized this construction by defining a rank-$k$ lift of a matroid $M$ given a rank-$k$ matroid $N$ on the set of circuits of $M$, and conjectured that all matroid lifts can be obtained in this way.
    In this sequel paper we simplify Walsh's construction
    and show that this conjecture is true for representable matroids but is false in general.
    This gives a new way to certify that a particular matroid is non-representable, which we use to construct new classes of non-representable matroids.

    Walsh also applied the new matroid lift construction to gain graphs over the additive group of a non-prime finite field, generalizing a construction of Zaslavsky for these special groups.
    He conjectured that this construction is possible on three or more vertices only for the additive group of a non-prime finite field.
    We show that this conjecture holds for four or more vertices, but fails for exactly three.
\end{abstract}

\section{Introduction and preliminaries}
Given matroids $M$ and $L$ on a common ground set $E$,
$L$ is a \emph{lift} of $M$ if there exists a matroid $K$ on ground set $E \cup F$ such that $M = K / F$ and $L = K \setminus F$.
If $L$ is a lift of $M$, then the rank of $L$ is at least the rank of $M$.
We say that $L$ is a \emph{rank-$k$ lift} of $M$ if the rank of $L$ is $k$ greater than that of $M$.
Rank-$1$ lifts, called \emph{elementary lifts}, are well understood.
Indeed, a classical theorem of Brylawski~\cite{brylawski1986},
which was previously stated in the dual by Crapo~\cite{Crapo}, 
says that the elementary lifts of a matroid $M$
are in bijection with the set of \emph{linear classes of circuits of $M$},
where a linear class is a set $\mathcal{C}$ of circuits satisfying the following:
\begin{center}
    if $C_1,C_2 \in \mathcal{C}$ and $|C_1 \cup C_2| - r_M(C_1 \cup C_2) = 2$,\\
    then each circuit $C$ of $M$ contained in $C_1 \cup C_2$ is also in $\mathcal{C}$.
\end{center}
We can state this bijection between linear classes of circuits and elementary lifts as follows.

\begin{thm}[{\cite{brylawski1986,Crapo}}] \label{brylawski}
Let $M$ be a matroid on ground set $E$ and let $\cC$ be a linear class of circuits of $M$. 
Then the function $r_{M'} \colon 2^E \to \mathbb Z$ defined, for all $X \subseteq E$, by 
$$r_{M'}(X)=\begin{cases}
r_M(X) & \text{ if each circuit of $M|X$ is in $\cC$},  \\
r_M(X)+1 & \text{ otherwise} 
\end{cases}$$
is the rank function of an elementary lift $M'$ of $M$.
Moreover, every elementary lift of $M$ can be obtained in this way.
\end{thm}

This raises the question of whether a similar characterization of higher-rank lifts is possible.
Walsh~\cite[Theorem 2]{walsh2022new} described a procedure that constructs a rank-$k$ lift of a matroid $M$
from a rank-$k$ matroid $N$ on the circuit set of $M$.
For this construction to work, $N$ must satisfy a particular constraint.
When $k = 1$, this constraint precisely says that the loops of $N$ are a linear class of circuits of $M$,
so Walsh's construction generalizes Theorem \ref{brylawski}.
Our first main result, stated below using the notation $\cl_N$ for the closure operator of a matroid $N$, is a simplification of Walsh's original construction in which we only require $N$ to satisfy a condition concerning pairs of circuits of $M$.

\begin{thm} \label{thm: new lift construction}
Let $M$ be a matroid on ground set $E$ and let $N$ be a matroid whose ground set is the circuit set of $M$.
Assume that if $C_1,C_2$ are circuits of $M$ for which $|C_1 \cup C_2| - r_M(C_1 \cup C_2) = 2$,
then each circuit $C$ of $M$ contained in $C_1 \cup C_2$ satisfies $C \in \cl_N(\{C_1, C_2\})$.
Then the function $r: 2^E \rightarrow \mathbb{Z}$ defined, for all $X \subseteq E$, by
$$r(X)=r_M(X)+r_N(\{C \colon \textrm{$C$ is a circuit of $M|X$}\})$$ 
is the rank function of a rank-$r(N)$ lift of $M$.
\end{thm}

Given a matroid $M$ and another matroid $N$ on the circuits of $N$ satisfying the hypothesis of Theorem~\ref{thm: new lift construction},
we write $M^N$ for the lift constructed in Theorem~\ref{thm: new lift construction}.
There are natural choices for a matroid $N$ satisfying the hypothesis of Theorem \ref{thm: new lift construction}, such as the derived matroids \cite{Jurrius,OxleyWang,Longyear}, matroids from gain graphs over certain groups \cite{walsh2022new}, and rank-$2$ uniform matroids.

It was conjectured in~\cite{walsh2022new} that every lift of $M$ is isomorphic to $M^N$ for some matroid $N$ on the circuits of $M$.
We prove that this true for representable matroids but false in general.
We use these facts to derive a new certificate for non-representability
which we then use to generate new families of non-representable matroids.
In particular, the following is our second main result.

\begin{thm} \label{counterexample}
For each integer $r \ge 5$ there is a rank-$r$ non-representable sparse paving matroid $K$ with a two-element set $X$ so that there is no matroid $N$ on the set of circuits of $K/X$ for which $(K/X)^N \cong K\del X$.
\end{thm}

This family of matroids may be of independent interest: they form an infinite antichain of non-representable sparse paving matroids that do not violate Ingleton's inequality.

Part of the motivation of~\cite{walsh2022new} was to generalize Zaslavsky's application of Theorem \ref{brylawski} to gain graphs.
A \emph{gain graph} is a pair $(G,\phi)$ where $G$ is a graph and $\phi$ is a \emph{gain function} that orientably labels the edges of $G$ by elements of a group $\Gamma$.
Zaslavsky \cite{zaslavsky1991liftmatroids} famously applied Theorem \ref{brylawski} to gain graphs by showing that for each gain function on a graph $G$ one can naturally construct a linear class $\mathcal B$ of circuits of $M(G)$, the graphic matroid of $G$.
The circuits in $\mathcal B$ are the \emph{balanced} cycles of $G$ with respect to the gain function, and the pair $(G,\mathcal B)$ is a \emph{biased graph}.
The elementary lift $M(G,\phi)$ of $M(G)$ obtained from applying Theorem \ref{brylawski} with the linear class $\mathcal B$ is the \emph{lift matroid} of $(G,\mathcal B)$, and it follows from Theorem \ref{brylawski} that a cycle of $G$ is a circuit of $M(G,\phi)$ if and only if it is balanced.
We direct the reader to \cite{group-labelings,walsh2022new,zaslavsky1991liftmatroids} for more background on gain graphs and lift matroids.

For certain groups Walsh \cite{walsh2022new} generalized this construction by defining a matroid $N$ on the circuits of $M(G)$ that satisfies the hypothesis of Theorem \ref{thm: new lift construction}.
To avoid technicalities, we only consider the \emph{full $\Gamma$-gain graph} graph $(K_n^{\Gamma}, \phi_n^{\Gamma})$ for a finite group $\Gamma$, where $K_n^{\Gamma}$ has vertex set $[n]$ and edge set ${[n]\choose 2}\times \Gamma$, and the gain function $\phi_n^{\Gamma}$ orients edge $(\{i,j\}, \gamma)$ from $i$ to $j$ when $i <j$ and assigns the label $\gamma$.
The previous version of Theorem~\ref{thm: new lift construction} was applied to gain graphs
to get the following, where $\mathbb{Z}_p^j$ denotes the direct sum of $j$ copies of the cyclic group of order $p$.

\begin{thm}[{\cite[Theorem 3]{walsh2022new}}] \label{groups construction}
Let $p$ be a prime, and let $n\ge 3$ and $j\ge 2$ be integers.
For each integer $i$ with $1\le i\le j$, there is a rank-$i$ lift $M$ of the cycle matroid $M(K_n^{\bZ_p^j})$ of $K_n^{\bZ_p^j}$ so that a cycle of $K_n^{\bZ_p^j}$ is a circuit of $M$ if and only if it is balanced.
\end{thm}

Surprisingly, it was shown that for finite abelian groups such a construction is only possible for groups of this form; namely, the additive group of a non-prime finite field.

\begin{thm}[{\cite[Theorem 4]{walsh2022new}}] \label{abelian groups converse}
Let $\Gamma$ be a nontrivial finite abelian group, and let $n\ge 3$ be an integer.
Let $M$ be a lift of the cycle matroid $M(K_n^{\Gamma})$ so that a cycle of $K_n^{\Gamma}$ is a circuit of $M$ if and only if it is balanced.
Then either $\Gamma\cong \bZ_p^j$ for some prime $p$ and integer $j\ge 2$, or $M$ is an elementary lift of $M(K_n^{\Gamma})$.
\end{thm}

It was conjectured in \cite[Conj. 25]{walsh2022new} that this holds more generally for every nontrivial finite group.
For our third and final main result we show that this conjecture is true when $n \ge 4$ and is false when $n = 3$.
A \emph{group partition} of a group $\Gamma$ is a partition of the non-identity elements of $\Gamma$
into sets $A_1,\dots,A_k$ such that each $A_i \cup \{\epsilon\}$ is a subgroup of $\Gamma$ for all $i \in [k]$, where $\epsilon$ is the identity element of $\Gamma$.
The partition is \emph{nontrivial} if it has more than one part.

\begin{thm} \label{all groups converse}
Let $\Gamma$ be a nontrivial finite group and let $n \ge 3$.
Let $\mathcal{M}_{n,\Gamma}$ be the class of lifts $M$ of $M(K_n^{\Gamma})$ so that a cycle of $K_n^{\Gamma}$ is a circuit of $M$ if and only if it is balanced.
Then $\mathcal{M}_{n,\Gamma}$ contains a non-elementary lift in precisely the following cases:
\begin{enumerate}
    \item $n = 3$ and $\Gamma$ has a nontrivial partition, or
    \item $n \ge 4$ and $\Gamma = \mathbb{Z}_p^j$ for some prime $p$ and $j \ge 2$.
\end{enumerate}
\end{thm}

We prove Theorem \ref{thm: new lift construction} in Section \ref{sec:simplified construction} and Theorem \ref{counterexample} in Section \ref{sec: the converse}.
One direction of Theorem \ref{all groups converse} was given in~\cite{walsh2022new} as Lemma~21;
we prove the converse direction in Section~\ref{sec: gain graphs}.
We follow the notation and terminology of Oxley \cite{oxley2006matroid}.

\section{A simplified construction} \label{sec:simplified construction}

We now state the original construction from \cite{walsh2022new} and then show that it is equivalent to Theorem \ref{thm: new lift construction}.
For a collection $\mathcal X$ of sets we write $\cup \mathcal X$ for $\cup_{X \in \mathcal X}X$.
A collection $\cC'$ of circuits of a matroid $M$ is \emph{perfect} if $|\cup \cC'| - r_M(\cup \cC') = |\cC'|$ and no circuit in $\cC'$ is contained in the union of the others.
Equivalently, $\cC'$ is contained in the collection of fundamental circuits with respect to a basis of $M$, because the set obtained from $\cup\cC'$ by deleting one element from each circuit that is not in any other circuit in $\cC'$ is independent in $M$.
Note that a pair $\{C_1, C_2\}$ is perfect if and only if $|C_1 \cup C_2| - r_M(C_1 \cup C_2) = 2$; 
in this case, we say that $\{C_1, C_2\}$ is a \emph{modular pair}.

In order to prove Theorem~\ref{thm: new lift construction},
we need to recall the lift construction from~\cite{walsh2022new}.

\begin{thm}[{\cite[Theorem 2]{walsh2022new}}] \label{old lift construction}
Let $M$ be a matroid on ground set $E$, and let $N$ be a matroid on the set of circuits of $M$ so that 
\begin{enumerate}[$(\ref{old lift construction}*)$]
\item if $\cC'$ is a perfect collection of circuits of $M$, then each circuit $C$ of $M$ contained in $\cup\cC'$ satisfies $C\in\cl_N(\cC')$.
\end{enumerate}
Then the function $r: 2^E \rightarrow \mathbb{Z}$ defined, for all $X \subseteq E$, by
$$r(X)=r_M(X)+r_N(\{C \colon \textrm{$C$ is a circuit of $M|X$}\})$$ 
is the rank function of a rank-$r(N)$ lift of $M$.
\end{thm}

Theorem \ref{thm: new lift construction} follows from Theorem \ref{old lift construction} and the following proposition, which shows that (\ref{old lift construction}$*$) and the hypothesis from Theorem \ref{thm: new lift construction} are equivalent.

\begin{prop} \label{simpler perfect}
Let $M$ be a matroid and let $N$ be a matroid on the circuits of $M$.
The following are equivalent
\begin{enumerate}
\item[$(*')$] For every modular pair $\{C_1, C_2\}$ of circuits of $M$, each circuit $C$ of $M$ contained in $C_1 \cup C_2$ satisfies $C \in \cl_N(\{C_1, C_2\})$.
\item[$(*)$] For every perfect collection $\cC'$ of circuits of $M$, each circuit $C$ of $M$ contained in $\cup\cC'$ satisfies $C\in \cl_{N}(\cC')$.
\end{enumerate}
\end{prop}
\begin{proof}
The implication $(*)\implies(*')$ is immediate. We now show $(*')\implies(*)$.
Assume $(*')$ and let $\cC'$ be minimal so that $(*)$ is false for $\cC'$.
Then $|\cC'| > 1$.
Let $C$ be a circuit of $M$ contained in $\cup\cC'$.
Every subset of $\cC'$ is a perfect collection of circuits of $M$;
we will make repeated tacit use of this fact.
If there is some $C' \in \cC'$ so that $C$ is contained in $\cup (\cC' - \{C'\})$, then $C \in \cl_N(\cC' - \{C'\})$ by the minimality of $\cC'$, and therefore $C \in \cl_N(\cC')$, as desired.
So for each $C'\in \cC'$, there is at least one element in $C'\cap C$ that is not in any other circuit in $\cC'$.
Let $X$ be a transversal of $\{(C'\cap C) - \cup (\cC' - \{C'\}) \colon C' \in \cC'\}$, and note that $X\subseteq C$.
Then 
\setcounter{equation}{0}
\begin{align}
|(\cup\cC')-X| &= |\cup \cC'| - |X| \\
&= r_{M}(\cup \cC') + |\cC'| - |X| \\
&= r_{M}(\cup \cC') \\
&= r_M((\cup\cC')-X),
\end{align}
so $(\cup\cC')-X$ is independent in $M$.
Line (2) holds because $\cC'$ is perfect, line (3) holds because $|X| = |\cC'|$, and line (4) holds because each element in $X$ is in a unique circuit of $M|(\cup \cC')$, so $X \subseteq \cl_M((\cup \cC') - X)$.

We claim that $M|(\cup\cC')$ is connected.
Otherwise, $C$ is contained in the union of a proper subset $\cC''$ of $\cC'$
since each each circuit of a matroid is contained in a single connected component.
Minimality of $\cC'$ then implies $C \in \cl_N(\cC'') \subseteq \cl_N(\cC')$.

Since $M|(\cup\cC')$ is connected and has corank at least two, there is a circuit $C_0$ of $M|(\cup\cC')$ so that $|C \cup C_0| - r_M(C \cup C_0) = 2$ and $C \cap C_0 \ne \varnothing$.
Let $x_1 \in X$.
Then the matroid $M|((C \cup C_0) - x_1)$ has corank one and thus contains a unique circuit $C_1$ (where $C_1 = C_0$ if $x_1 \notin C_0$).
The circuit $C_1$ contains some $x_2 \in X$, since $(\cup\cC')-X$ is independent in $M$.
Let $C_2$ be the unique circuit of $M|((C \cup C_0) - x_2)$.
The circuits $C,C_1, C_2$ are distinct because $x_1,x_2 \in C$, while $C_1$ contains $x_2$ but not $x_1$ and $C_2$ does not contain $x_2$.
Then $|C_1 \cup C_2| - r_M(C_1 \cup C_2) \ge 2$, and since $|C \cup C_0| - r_M(C \cup C_0) = 2$,
this implies that $C_1 \cup C_2 = C \cup C_0$. In particular, $|C_1 \cup C_2| - r_M(C_1 \cup C_2) = 2$ and $C \subseteq C_1 \cup C_2$.
Then by $(*')$ we have $C \in \cl_N(\{C_1, C_2\})$.

We will finish the proof by using the minimality of $\cC'$ to show that $C_1, C_2 \in \cl_N(\cC')$.
For each $i \in \{1,2\}$, let $C'_i$ be the unique circuit in $\cC'$ that contains $x_i$.
No element in $C_i' - \cup (\cC'-\{C_i'\})$ is in a circuit of $M|((\cup \cC') - x_i)$, because $\cC' - \{C_i'\}$ is a perfect collection of circuits with $|\cup (\cC' - \{C_i'\})| - r_M(\cup (\cC' - \{C_i'\})) = |\cC'| - 1$.
Since $C_i$ does not contain $x_i$, this implies that $C_i \subseteq \cup(\cC' - \{C_i'\})$.
Since $|\cC' - \{C_i'\}| < |\cC'|$, the minimality of $\cC'$ implies that $C_i \in \cl_N(\cC' - \{C_i'\})$ and is thus in $\cl_N(\cC')$.
Finally, since $C_1, C_2 \in \cl_N(\cC')$ and $C \in \cl_N(\{C_1, C_2\})$, it follows that $C \in \cl_N(\cC')$.
\end{proof}

One advantage of $(*')$ over $(*)$ is that it is a local condition rather than a global condition, and therefore may be easier to verify for certain choices of $N$.
For example, when $M$ is graphic it suffices to check condition $(*')$ only when $C_1$ and $C_2$ are in a common theta subgraph.

\section{The converse} \label{sec: the converse}

It was conjectured in \cite[Conj. 1.6]{walsh2022new} that the converse of Theorem \ref{thm: new lift construction} holds: for every matroid $K$ with a set $X$, there is a matroid $N$ on the circuits of $K/X$ so that $(K/X)^N \cong K \del X$.
We show that this is true if $K$ is representable but false in general, even when $|X| = 2$.

\begin{prop}\label{prop: representable lifts}
    Let $K$ be an $\mathbb F$-representable matroid and let $X$ be a subset of its ground set.
    Then there exists an $\mathbb F$-representable matroid $N$ on the circuits of $K/X$ such that $(K/X)^N \cong K \del X$.
\end{prop}
\begin{proof}
    Denote $M:= K / X$ and $L := K \setminus X$ and
    let $E$ denote the ground set of $M$ and $L$.
    Without loss of generality we may assume that $X$ is independent in $K$.
    Therefore there exists a matrix $A$ whose column-matroid is $K$
    such that the columns corresponding to elements of $X$ are distinct standard basis vectors.
    Thus one obtains an $\mathbb{F}$-representation $A_M$ of $M$ by deleting the columns corresponding to $X$,
    as well as the rows corresponding to the nonzero entries of these columns.
    One obtains an $\mathbb{F}$-representation $A_L$ of $L$ by deleting the $X$ columns.
    From this, one obtains an $\mathbb{F}$-representation $A_M$ of $M$ by deleting the rows
    where the columns corresponding to $X$ have their nonzero entries.
    
    For each circuit $C$ of $M$,
    let $x_C$ be an element of the kernel of $A_M$ whose support is $C$
    and let $B$ be a matrix whose column-set is the following.
    \[
        \{A_L x_C : C \ \text{is a circuit of } M\}.
    \]
    Let $N$ be the column matroid of $B$, which we can view as a matroid whose ground set is the circuit set of $M$.
    Let $C_1$ and $C_2$ be circuits of $M$ with $|C_1 \cup C_2| - r_M(C_1 \cup C_2) = 2$, and let $C$ be a circuit of $M$ with $C \subseteq C_1 \cup C_2$.
    Consider the matrix obtained by restricting $A_M$ to the columns indexed by $C_1 \cup C_2$.
    It has a two-dimensional kernel with a basis obtained from $\{x_{C_1},  x_{C_2}\}$ by deleting entries corresponding to elements of $E$ outside $C_1 \cup C_2$.
    Therefore $x_C = \alpha x_{C_1} + \beta x_{C_2}$ for some $\alpha,\beta \in \mathbb{F}$
    and $A_Lx_C = \alpha A_Lx_{C_1} + \beta A_Lx_{C_2}$.
    In particular, $C \in \cl_N(\{C_1, C_2\})$,
    which allows us to apply the construction given in Theorem~\ref{thm: new lift construction}.

    It remains to prove that $L = M^N$.
    Let $Y \subseteq E$, let $T \subseteq Y$ be a basis of $M|Y$, and for each $e \in Y \setminus T$ let $C_e$
    denote the unique circuit of $M$ in $T \cup \{e\}$.
    If $Y$ is dependent in $M^N$, then $T$ is a proper subset of $Y$ and there exists a nontrivial linear dependence of the following form
    \[
        \sum_{e \in Y \setminus T} \lambda_e A_Lx_{C_e} = 0.
    \]
    Then $\sum_{e \in Y \setminus T} \lambda_e x_{C_e}$ lies in the kernel of $A_L$.
    It is nonzero since $x_{C_e}$ is zero at all $f \in Y\setminus(T \cup \{e\})$ and nonzero at $e$.
    So $Y$ is also dependent in $L$.
    
    Now assume $Y$ is dependent in $L$ and let $y$ be such that $A_L y = 0$.
    Then $A_M y = 0$.
    Define $k := |Y| - r_M(Y)$ and note that this is the dimension of the kernel of the matrix $D$ obtained from $A_L$
    by restricting to columns corresponding to $Y$.
    Since the kernel of every matrix is spanned by its support-minimal elements,
    there exists a set of circuits $\{C_1,\dots,C_k\}$ of $M$ so that
    $\{x_{C_1},\dots,x_{C_k}\}$ is a basis of the nullspace of $D$, modulo adding/removing entries corresponding to elements of $E \setminus Y$.
    This gives us scalars $\lambda_1,\dots,\lambda_k$ so that
    \[
        y = \sum_{i = 1}^k \lambda_i x_{C_i}.
    \]
    Multiplying both sides of the above on the left by $A_L$ tells us that $\{C_1,\dots,C_k\}$ is dependent in $N$.
    Since $\{x_{C_1},\dots,x_{C_k}\}$ spans the nullspace of $D$,
    $r_N(\{C_1,\dots,C_k\}) < k$.
    But then $|Y| > r_M(Y) + r_N(\{C_1,\dots,C_k\})$ so $Y$ is also dependent in $M^N$.
\end{proof}

If $M$ and $L$ are representable over a field $\mathbb{F}$ and $L$ is a lift of $M$, this does \emph{not} imply that there is an $\mathbb F$-representable matroid $K$ so that $L = K \del X$ and $M = K/X$ for some set $X \subseteq E(K)$.
Consider the following example.
Given nonnegative integers $r \le n$, the uniform matroid of rank $r$ on $n$ elements is denoted $U_{r,n}$.
Let $M = U_{1,3}$ and let $L = U_{2,3}$.
Then both $M$ and $L$ are representable over $\mathbb{F}_2$.
Let $K$ be a matroid on ground set $E \cup \{e_0\}$.
If $K \del e_0 = L$ and $K / e_0 = M$, then $K = U_{2,4}$ which is not representable over $\mathbb{F}_2$.

We next construct an infinite family of matroids for which the converse of Theorem~\ref{thm: new lift construction} does not hold.
For certain values of $r$ and $t$,
we will define a set of $r$-element subsets $\cC(r,t)$ of $[2t+2]$
and then show that they are the circuit-hyperplanes of a matroid of rank $r$ on ground set $[2t+2]$
which we will denote $K(r,t)$.
We will then show that there is no matroid $N$ on the circuit set of $K(r,t)/\{2t+1,2t+2\}$ such that $K(r,t)\setminus \{2t+1,2t+2\} = (K(r,t)/\{2t+1,2t+2\})^N$.
Proposition~\ref{prop: representable lifts} will then imply that $K(r,t)$ is not representable over any field.

\begin{defn}
    Let $r \ge 4$ and $t \ge 3$ be integers satisfying $r \le 2t-2$.
    For $i = 1,\dots, t$, let $C_i \subseteq [2t]$ be defined as
    \[
        C_i := \{1+2(i-1), 2 + 2(i - 1), \dots, (r-2) + 2(i-1)\}
    \]
    with all numbers taken modulo $2t$.
    Then define:
    \begin{enumerate}
        \item $X := \{2t+1,2t+2\}$,
        \item $\cC'(r,t) := \{C_i \cup X \colon i \in [t]\}$, and
        \item $\cC''(r,t) := \{C_i \cup C_{i+1} \colon i \in [t - 1]\}$.
    \end{enumerate}
    Define $\cC(r,t)$ to be the set of subsets of $[2t+2]$ containing $\cC'(r,t) \cup \cC''(r,t)$
    and all $(r+1)$-element subsets that do not contain an element of $\cC'(r,t)$ or $\cC''(r,t)$.
\end{defn}



We will soon see that $\cC(r,t)$ is the circuit set of a matroid, but before doing this, we look at the case of $r=4$ and $t=3$.
Here we have
\begin{enumerate}
    \item $C_1 = \{1,2\}$, $C_2 = \{3,4\}$, $C_3 = \{5,6\}$, and $X = \{7,8\}$
    \item $\cC'(4,3) = \{1278,3478,5678\}$
    \item $\cC''(4,3) = \{1234,3456\}$.
\end{enumerate}
In particular, $\cC'(4,3) \cup \cC''(4,3)$ is the set of circuit-hyperplanes of the V\'amos matroid $V_8$.
So $K(r,t)$ is a generalization that captures the cyclic nature of the set of circuit-hyperplanes of $V_8$.
Recall that a matroid of rank $r$ is \emph{sparse paving} if every $r$-element subset is either a basis or a circuit-hyperplane.

\begin{prop}
Let $r \ge 4$ and $t \ge 3$ be integers satisfying $r \le 2t-2$.
Then $\cC(r,t)$ is the circuit set of a rank-$r$ sparse paving matroid $K(r,t)$ on ground set $[2t + 2]$.
\end{prop}
\begin{proof}
It suffices to show that no two sets in $\cC'(r,t) \cup \cC''(r,t)$ intersect in $r-1$ elements.
We have three cases to consider.

{Case 1:} Let $C \in \cC'(r,t)$ and $C' \in \cC''(r,t)$.
Then $|C \cap C'| \le r - 2$ because $C$ contains $\{2t + 1, 2t + 2\}$ and $C'$ does not.

{Case 2:} Let $C, C' \in \cC'(r,t)$, so $C = C_i \cup X$ and $C' = C_j \cup X$ for some $i,j \in [t]$.
Since $r \le 2t-2$ it follows that $|C_k \cap C_{k+1}| = r - 4$ for all $k$ (indices taken modulo $t$) and $|C_i \cap C_j| \le r - 4$.
Then $|C \cap C'| = |C_i \cap C_{i+1}| + 2 \le r - 2$, as desired.

{Case 3:} Let $C, C' \in \cC''(r,t)$, so $C = C_i \cup C_{i+1}$ and $C' = C_j \cup C_{j+1}$ for some $i, j\in [t]$.
Then $|C \cap C'|$ is maximized when $j= i + 1$, so we may assume that $C' = C_{i+1} \cup C_{i+2}$, taking indices modulo $t$.
Since $r \le 2t-2$ it follows that $C_i \cap C_{i+2}  \subseteq C_{i+1}$.
This implies that $C \cap C' \subseteq C_{i+1}$ and so it follows that $|C \cap C'| \le |C_{i+1}| = r - 2$.
\end{proof}

The following implies Theorem \ref{counterexample}.

\begin{thm} \label{thm: converse if false}
Let $r \ge 4$ and $t \ge 3$ be integers satisfying $r \le 2t-2$.
There is no matroid $N$ on the circuits of $K(r,t)/\{2t + 1, 2t + 2\}$ for which $K(r,t)\del \{2t + 1, 2t + 2\}$
is isomorphic to $(K(r,t)/\{2t + 1, 2t + 2\})^N$.
Moreover, $K(r,t)$ is not representable over any field.
\end{thm}
\begin{proof}
Denote $K := K(r,t)$ and $X := \{2t + 1, 2t + 2\}$.
Let $M := K/X$ and $L := M \del X$, so $L$ is a rank-$2$ lift of $M$.
For each $i \in [t]$, the set $C_i$ is a circuit of $M$ and is independent in $L$.
So for each $i \in [t]$ the set $C_i$ is a non-loop of $N$.
We will argue that the following statements hold for $M$ and $L$:
\begin{enumerate}[$(a)$]
\item $(C_i , C_{i+1})$ is a modular pair of circuits of $M$ for each $i \in [t-1]$,

\item $(C_1, C_{t})$ is a modular pair of circuits of $M$,

\item $r_L(C_i \cup C_{i+1}) - r_M(C_i \cup C_{i+1}) = 1$ for each $i \in [t-1]$,

\item $r_L(C_1 \cup C_{t}) - r_M(C_1 \cup C_{t}) = 2$.
\end{enumerate}

For $(a)$, note that $|C_i \cup C_{i+1}| = r$ and $r_M(C_i \cup C_{i+1}) = r - 2$, because $r_K(C_i \cup C_{i+1} \cup X) = r$.
The same argument also proves $(b)$.
For $(c)$, note that $r_K(C_i \cup C_{i+1}) = r - 1$ because $C_i \cup C_{i+1}$ is a circuit of $K$ of cardinality $r$.
However $r_L(C_1 \cup C_{t}) - r_M(C_1 \cup C_{t}) = 2$ because $C_1 \cup C_{t}$ is independent in $K$, proving $(d)$.

Now suppose there is a matroid $N$ on the circuits of $M$ so that $M^N \cong L$.
Then $(a)$ and $(c)$ together imply that $C_i$ and $C_{i+1}$ are parallel in $N$ for all $i \in [t]$, which implies that $C_1$ and $C_{t}$ are parallel in $N$.
But $(b)$ and $(d)$ together imply that $C_1$ and $C_{t}$ are independent in $N$, a contradiction.
It now follows from Proposition \ref{prop: representable lifts} that $K$ is not representable.
\end{proof}

As illustrated by Theorem \ref{thm: converse if false}, the fact that the converse of Theorem \ref{thm: new lift construction} is false in general but true for representable matroids gives a new way to certify non-representability.
We hope that this method is in fact `new' and that $K(r,t)$ cannot be certified as non-representable using existing means.
However, there are many ways to certify non-represent\-ability, and we make no attempt to test $K(r,t)$ against them all.
We merely show that $K(r,t)$ does not violate the most well-known certificate for non-representability: Ingleton's inequality.
Ingleton \cite{ingleton1971} proved that if a matroid has sets $A,B,C,D$ so that
\begin{align*}
&r(A \cup B) + r(A \cup C) + r(A \cup D) + r(B \cup C) + r(B \cup D)\\
& \ge r(A) + r(B) + r(A \cup B \cup C) + r(A \cup B \cup D) + r(C \cup D)
\end{align*}
then it is not representable.
We show that this cannot be used to certify non-represent\-ability of $K(r,t)$ when $r \ge 5$ and $r \le 2t-3$.
Following \cite{PeterJorn}, we say that a matroid is \emph{Ingleton} if any choice of four subsets satisfies the above inequality.

\begin{prop}
Let $r \ge 5$ and $t \ge 3$ be integers satisfying $r \le 2t-3$.
Then $K(r,t)$ is Ingleton.
\end{prop}
\begin{proof}
Nelson and van der Pol \cite[Lemma 3.1]{PeterJorn} showed that a rank-$r$ sparse paving matroid is Ingleton if and only if there are no pairwise disjoint subsets $I, P_1, P_2, P_3, P_4$ so that $|I| = r - 4$ and $|P_i| = 2$ for all $i \in \{1,2,3,4\}$, while $I \cup P_i \cup P_j$ is a circuit of for all $\{i,j\} \ne \{3,4\}$ and $I \cup P_3 \cup P_4$ is a basis.
Suppose that such sets exist for $K(r,t)$.
Then $I$ is an $(r-4)$-element set contained in at least five circuit-hyperplanes of $K(r,t)$, and
$Y = I \cup P_1 \cup P_2 \cup P_3 \cup P_4$ is an $(r + 4)$-element set that contains at least five circuit-hyperplanes of $K(r,t)$.
Now let $X = \{2t+1, 2t+2\}$.

If $I \ne C_i \cap C_{i+1}$ for some $i$ (indices taken modulo $t$), then $I$ is in at most one circuit-hyperplanes of the form $C_j \cup X$, and at most three of the form $C_j  \cup C_{j+1}$, a contradiction.
So $I = C_i \cap C_{i+1}$ for some $i$.
Then $I$ is contained in exactly five circuit-hyperplanes, namely $C_i \cup X$, $C_{i+1} \cup X$, $C_{i-1} \cup C_{i}$, $C_i \cup C_{i+1}$, and $C_{i+1} \cup C_{i+2}$, which means that each of these sets is contained in $Y$.
However, $r \le 2t-3$ implies that $|C_{i-1} \cup C_i \cup C_{i+1} \cup C_{i+2} \cup X| > r + 4$, a contradiction.
\end{proof}

It follows from Theorem \ref{thm: converse if false} and results of Nelson and van der Pol \cite{PeterJorn} that $K(r,t)$ also cannot be certified as non-representable via a small non-representable minor.
Following \cite{PeterJorn}, a rank-$4$ sparse paving matroid $M$ is \emph{V\'amos-like} if it has a partition $(P_1,P_2,P_3,P_4)$ such that exactly five of the six pairs $P_i \cup P_j$ form circuits of $M$.
There are $39$ V\'amos-like matroids, one of which is the V\'amos matroid itself, and none are representable.
They prove that a sparse paving matroid is Ingleton if and only if it has no V\'amos-like minor.
So, Theorem \ref{thm: converse if false} implies that $K(r,t)$ has no V\'amos-like minor when $r \ge 5$ and $r \le 2t-3$.

While $K(r,t)$ is non-representable, we conjecture that it is very close to being representable, in the following sense.


\begin{conj} \label{conj: relax}
For all integers $r$ and $t$ with $r \ge 4$, $t \ge 3$ and $2t +2 \ge  r + 4$, any matroid obtained from $K(r,t)$ by relaxing a circuit-hyperplane into a basis is representable.
\end{conj}

When $r \le 2t-3$, $K(r,t)$ has no element in every circuit-hyperplane, so Conjecture~\ref{conj: relax} would imply that $K(r,t)$ is an excluded minor for the class of representable matroids.
We prove one more interesting property of $K(r,t)$.

\begin{prop} \label{antichain}
Let $r \ge 4$ and $t \ge 3$ be integers satisfying $r \le 2t-3$
and let $M$ be a proper minor of $K(r,t)$.
Then $M$ is not isomorphic to $K(r',t')$ for any integers $r'\ge 4$
and $t' \ge 3$ satisfying $r' \le 2t'-3$.
\end{prop}
\begin{proof}
Suppose $K(r',t')$ is a minor of $K(r,t)$ and $(r',t') \ne (r,t)$.
Then $r' < r$.
Let $A \subseteq E(K(r,t))$ be independent so that $K(r,t)/A$ has a spanning $K(r',t')$-restriction, so $|A| = r - r'$.
If $C$ is a circuit of $K(r',t')$, then $C \cup A$ contains a circuit of $K(r,t)$.
If $C$ is a circuit-hyperplane of $K(r',t')$ then $|C \cup A| = r' + (r - r') = r$, which implies that $C \cup A$ is a circuit-hyperplane of $K(r,t)$.
Since $K(r',t')$ has $2t' - 1$ circuit-hyperplanes, this implies that $A$ is contained in at least $2t' - 1$ circuit-hyperplanes of $K(r,t)$.

It is straightforward to show that the intersection of any $h$ of the sets $C_i$ has size at most $r - 2h$.
So if $A$ is contained in $h$ of the sets $C_i$ and $h > r'/2$, then $|A| \le r - 2h < r - r' = |A|$, a contradiction.
So $A$ is contained in $C_i$ for at most $r'/2$ different choices of $i$.
Then $A$ is contained in at most $r'/2$ circuit-hyperplanes of $K(r,t)$ of the form $C_i \cup X$ and at most $r'/2 + 1$ circuit-hyperplanes of the form $C_i \cup C_{i+1}$.
But then $A$ is contained in at most $r'/2 + (r'/2 + 1) = r' + 1$ circuit-hyperplanes of $K(r,t)$.
Since $r' + 1 < 2t' - 1$ when $r'\le 2t' - 3$, this is a contradiction.
\end{proof}

So $\{K(r,t)\colon r\ge 4, t\ge 3, r \le 2t-3\}$ is an infinite antichain of non-representable matroids that all satisfy Ingleton's inequality, and if Conjecture \ref{conj: relax} is true then each is also an excluded minor for representability.
We comment that while $K(r,t)$ is constructed using rank-$2$ lifts, related constructions using rank-$t$ lifts with $t > 2$ are likely possible as well.
We conclude this section with the following question.

\begin{ques}
Which matroids $K$ have the property that for every set $X \subseteq E(K)$ there is a matroid $N$ on the circuits of $K/X$ such that $(K/X)^N \cong K \del X$?
\end{ques}

For example, if every algebraic matroid has this property then $K(r,t)$ would be non-algebraic for all $r \ge 4$ and $t \ge 3$.
This may be a promising direction since the V\'amos matroid is non-algebraic \cite{IngletonMain1975} and isomorphic to $K(4,3)$.

\section{Gain graphs} \label{sec: gain graphs}
Recall that $(K_n^{\Gamma}, \phi_n^{\Gamma})$ is the gain graph over a finite group $\Gamma$ where $K_n^{\Gamma}$ has vertex set $[n]$ and edge set ${[n]\choose 2}\times \Gamma$, and the gain function $\phi_n^{\Gamma}$ orients edge $(\{i,j\}, \alpha)$ from $i$ to $j$ when $i <j$ and assigns the label $\alpha$.
We write $\alpha_{ij}$ for the edge $(\{i,j\}, \alpha)$, for convenience.
For each $\alpha \in \Gamma$ we write $E_\alpha$ for $\{(\{i,j\},\alpha) \colon 1 \le i < j \le n\}$; these are the edges \emph{labeled} by $\alpha$.
For a set $A \subseteq \Gamma$ we write $E_A$ for $\cup_{\alpha \in A} E_\alpha$.
A cycle of $(K_n^{\Gamma}, \phi_n^{\Gamma})$ is \emph{balanced} if an oriented product of its edge labels is equal to the identity element of $\Gamma$.
For the remainder of this section we shall refer to balanced and unbalanced cycles of $K_n^{\Gamma}$, with the gain function $\phi_n^{\Gamma}$ implicit.

We can use $\Gamma$ to define special automorphisms of the graph $K_n^{\Gamma}$, as follows.
Given an integer $k \in [n]$ and an element $\beta \in \Gamma$, define an automorphism $f_\beta \colon E(K_n^{\Gamma}) \to E(K_n^{\Gamma})$ by 
\[
    f_\beta(\alpha_{ij}) = \begin{cases}
        (\beta^{-1} \cdot \alpha)_{ij} \quad \textnormal{if   } i = k \\
        (\alpha\cdot \beta)_{ij} \quad \textnormal{if   } j = k \\
        \alpha_{ij} \quad \textnormal{otherwise}.
    \end{cases}
\]
For each edge $e$ of $K_n^{\Gamma}$ we say that $f_\beta(e)$ is obtained from $e$ by \emph{switching} at vertex $k$ with value $\beta$.
If a set $X$ of edges of $K_n^{\Gamma}$ can be obtained from a set $Y$ via a sequence of switching operations we say that $X$ and $Y$ are \emph{switching equivalent}.
It is straightforward to check that switching maps balanced cycles to balanced cycles and unbalanced cycles to unbalanced cycles.
We comment that switching is typically an operation on gain functions, and our application of this operation to define a graph automorphism is nonstandard.

For each nontrivial finite group $\Gamma$ and integer $n\ge 3$, we define $\cM_{n, \Gamma}$ to be the class of lifts of the cycle matroid $M(K_n^{\Gamma})$ for which a cycle of $K_n^{\Gamma}$ is a circuit of $M$ if and only if it is balanced.
Each matroid in $\cM_{n, \Gamma}$ is simple, since each $2$-element cycle of $M(K_n^{\Gamma})$ is unbalanced.

We now generalize Theorem \ref{abelian groups converse} in the case that $n \ge 4$.

\begin{thm}\label{thm: 4 or more}
Let $n\ge 4$ be an integer, let $\Gamma$ be a finite group, and let $M \in \cM_{n,\Gamma}$.
If $r(M) - r(M(K_n^{\Gamma})) > 1$, then there is a prime $p$ and an integer $j\ge 2$ so that $\Gamma\cong \bZ_p^j$.
\end{thm} 
\begin{proof}
It suffices to prove that $\Gamma$ is abelian; then Theorem \ref{abelian groups converse} implies that $\Gamma\cong \bZ_p^j$ for some prime $p$ and integer $j \ge 2$.
Let $\epsilon$ denote the identity element of $\Gamma$.
It was proved in \cite[Lemma 19]{walsh2022new} that each $\alpha \in \Gamma - \{\epsilon\}$ satisfies $r_M(E_{\{\alpha, \ep\}}) = n$.
For $\alpha,\beta\in\Gamma-\{\ep\}$, we write $\alpha\sim\beta$ if $r_M(E_{\{\alpha,\beta,\ep\}})=n$.
By \cite[Lemma 21]{walsh2022new}, $\sim$ is an equivalence relation, and for each equivalence class $A$ we have that $r_M(E_{A \cup \ep}) = n$ and $A \cup \ep$ is a subgroup of $\Gamma$.
Let $\cA$ denote the set of equivalence classes under $\sim$.
Then $\cA$ is a group partition of $\Gamma$, and $|\cA|\ge 2$ because $r(M) - r(M(K_n^{\Gamma})) > 1$.

\begin{claim} \label{abelian}
    If $\alpha, \beta \in \Gamma$ commute whenever $\alpha$ and $\beta$ are in different sets in $\cA$, then $\Gamma$ is abelian.
\end{claim}
\begin{proof}
Let $\alpha \in \Gamma$, and let $A \in \cA$ contain $\alpha$.
Since $A \cup \ep$ is a proper subgroup of $\Gamma$, we have $|A|< |\Gamma|/2$, and so $|\Gamma - A| > |\Gamma|/2$.
The centralizer of $\alpha$ contains $\Gamma - A$, and thus contains more than $|\Gamma|/2$ elements.
Since the centralizer of $A$ is a subgroup of $\Gamma$, it follows that it is equal to $\Gamma$.
Thus, $\alpha$ commutes with every element of $\Gamma$.
Since the same argument applies to every element of $\Gamma$, it follows that $\Gamma$ is abelian.
Now that $\Gamma$ is abelian the theorem statement follows from Theorem \ref{abelian groups converse}.
\end{proof}

Suppose the theorem is false, and let $\Gamma$ be a finite non-abelian group with $|\Gamma|$ minimal so that the theorem is false for $\Gamma$.
We first show that $\Gamma$ has a $2$-element generating set.
We may assume that there are distinct sets $A_1$ and $A_2$ in $\cA$ and elements $a_1 \in A_1$ and $a_2 \in A_2$ that do not commute, or else the theorem statement holds by Claim \ref{abelian}.
Let $\Gamma'$ be the subgroup of $\Gamma$ generated by $\{a_1, a_2\}$, and let $M' = M|E_{\Gamma'}$. 
Since $a_1 \nsim a_2$, $r(M') - r(M(K_n^{\Gamma'})) > 1$.
Since $M'$ is a restriction of $M$, a cycle of $K_n^{\Gamma'}$ is a circuit of $M'$ if and only if it is balanced.
If $\Gamma' \ne \Gamma$, then the minimality of $|\Gamma|$ implies that $\Gamma'$ is abelian.
But then $a_1$ and $a_2$ commute, a contradiction.
So $\{a_1, a_2\}$ is a $2$-element generating set of $\Gamma$.
It follows from \cite[Lemma 20]{walsh2022new} that $E_{\{a_1, a_2, \ep\}}$ spans $M$, and it follows from the submodularity of $r_M$ that $r_M(E_{\{a_1, a_2, \ep\}}) = n + 1$.
So $M$ is a rank-$2$ lift of $M(K_n^{\Gamma})$, because the cycle matroid of a connected $n$-vertex graph has rank $n-1$.

Suppose for a contradiction that there is some $A \in \cA$ so that $A \cup \epsilon$ is not a normal subgroup of $\Gamma$.
We will define a pair of hyperplanes of $M$ that violate the hyperplane axioms.
Let $H_1 = E_{A \cup \epsilon}$.
Note that $H_1$ is a hyperplane of $M$, by the definition of $\sim$.
Let $\{V, V'\}$ be a partition of $[n]$ with $|V|, |V'| \ge 2$, and let $H_2$ be the set of edges with both ends in the same part.
Then $H_2$ is a flat of $M(K_n^{\Gamma})$, and is therefore a flat of $M$ because $M$ is a lift of $M(K_n^{\Gamma})$ (see \cite[Prop. 7.3.6]{oxley2006matroid}).
This implies that $r_M(H_2) \le r(M) - 1$ because $H_2$ does not include all elements of $M$.
Let $f$ be an edge labeled by $\epsilon$ and with ends $v \in V$ and $v' \in V'$.
Then each edge that has one end in $V$ and the other end in $V'$ but does not have ends $v$ and $v'$ is contained in a balanced cycle with three or four edges and all but one edge in $H_2 \cup f$.
This implies that $H_2 \cup f$ spans all edges that do not have ends $v$ and $v'$, and it follows that $H_2 \cup f$ spans $M$.
So $r_M(H_2) \ge r(M) - 1$ and therefore $H_2$ is a hyperplane of $M$.
Since $A \cup \ep$ is not normal, there is some element $b \in \Gamma$ so that $b^{-1}\cdot A \cdot b \ne A$.
Let $a \in A$ so that $b^{-1} \cdot a \cdot b = c$ for some $c \notin A \cup \epsilon$.
Consider the set $B = (H_1 \cap H_2) \cup e$, where $e$ is an edge with one end in $V$ and the other in $V'$, with label $b$, and directed from $V$ to $V'$.
By the hyperplane axioms, $B$ is contained in a hyperplane of $M$.
However, we claim that $B$ spans $M$.
Let $B'$ be obtained from $B$ by switching with value $b$ at each vertex in $V$.
In $B'$, the edge $e$ is labeled by $\epsilon$, and the set of labels of edges between each pair of vertices in $V$ is $b^{-1} \cdot (A \cup \epsilon) \cdot b$.
Note that $B'$ has a spanning tree of edges all labeled by $\epsilon$.
Then the set obtained from $B'$ by taking closure under balanced cycles contains $E_{\{\epsilon, a, c\}}$, which spans $M$ because $a \nsim c$.
Since switching preserves balanced cycles, this implies that $B$ also spans $M$, a contradiction.

We have shown that $A \cup \ep$ is a normal subgroup of $\Gamma$ for each $A \in \cA$.
Since normal subgroups that intersect in the identity commute, it follows that each $\alpha \in A$ commutes with every element in $\Gamma - A$.
Thus, Claim \ref{abelian} implies that $\Gamma$ is abelian, which is a contradiction.
\end{proof}

Surprisingly, Theorem \ref{abelian groups converse} does not generalize to non-abelian groups when $n = 3$.
Walsh showed that if there exists $M \in \mathcal{M}_{n,\Gamma}$ whose rank is at least two greater than that of $M(K_n^\Gamma)$,
then $\Gamma$ has a nontrivial partition~\cite[Lemma~21]{walsh2022new}.
Theorem~\ref{thm: partition implies higher rank lifts} establishes the converse for $n = 3$.
Recall that a \emph{nontrivial partition} of a group $\Gamma$ with identity $\ep$ is a partition $\cA$ of $\Gamma - \{\ep\}$ so that $A \cup \ep$ is a subgroup of $\Gamma$ for all $A \in \cA$, and $|\cA| \ge 2$.
For example, $\bZ_p^j$ has a nontrivial partition into $\frac{p^j - 1}{p - 1}$ copies of the cyclic group $\bZ_p$, and the dihedral group has a nontrivial partition where each reflection is a part and the nontrivial rotations form a part.
We direct the reader to \cite{zappa2003partitions} for background on group partitions.

A group may have multiple nontrivial partitions.
That said, each finite group with a nontrivial partition has a canonical nontrivial partition called the \emph{primitive partition},
first described in~\cite{young1927partitions}.
It is universal in the following sense: if $\mathcal{A}$ is the primitive partition of a finite group $G$ and
$\mathcal{B}$ is a nontrivial partition of $G$, then for each $B \in \mathcal{B}$,
the following is a partition of $B \cup \epsilon$:
\[
    \{A \in \mathcal{A} : A \cup \epsilon \textnormal{ is a subgroup of } B\cup \epsilon\}.
\]
For our purposes, the most important property of the primitive partition is the following.
\begin{prop}[\cite{baer1961partitionen,zappa2003partitions}]\label{prop: primitive partition is normal}
    Let $\Gamma$ be a finite group with a nontrivial partition, and let $\mathcal{A}$ be the primitive partition of $\Gamma$.
    If $A \in \mathcal{A}$ and $\gamma \in \Gamma$, then $\gamma \cdot A \cdot \gamma^{-1} \in \mathcal{A}$.
\end{prop}

We next prove a partial generalization of Theorem \ref{groups construction} to non-abelian groups in the case that $n = 3$.

\begin{thm}\label{thm: partition implies higher rank lifts}
Let $\Gamma$ be a finite group with a nontrivial partition.
Then there is a rank-$2$ lift $M$ of $M(K_3^{\Gamma})$ so that a cycle of $K_3^{\Gamma}$ is a circuit of $M$ if and only if it is balanced.
\end{thm}
\begin{proof}
Let $\epsilon$ be the identity element of $\Gamma$, and let $\mathcal A$ be the primitive partition of $\Gamma$.
We define a collection $\cH$ of subsets of the edges of $K_3^{\Gamma}$ where $H \in \cH$ if
\begin{enumerate}[(a)]
\item $H$ is switching equivalent to $E_{A \cup \{\epsilon\}}$ for some $A \in \mathcal A$, or

\item $H$ consists of all edges between some pair $i,j \in \{1,2,3\}$.
\end{enumerate}

Note that $\cH$ is invariant under switching, and also relabeling vertices.
We need the following fact about nontrivial partitions.

\begin{claim} \label{group theory}
Let $A \in \cA$ and let $\alpha, \beta\in \Gamma$ so that $\epsilon \in \alpha \cdot (A\cup \epsilon) \cdot \beta$.
Then $\alpha \cdot (A\cup \epsilon) \cdot \beta = A' \cup \epsilon$ for some $A' \in \cA$.
\end{claim}
\begin{proof}
There is some $\gamma \in A \cup \ep$ so that $\alpha \cdot \gamma \cdot \beta = \epsilon$,
i.e.~so that $\beta^{-1} = \alpha\cdot \gamma$.
Since $A \cup \epsilon$ is a subgroup of $\Gamma$,
$\gamma \cdot (A \cup \epsilon) = A \cup \epsilon$ and therefore
$\alpha \cdot (A \cup \epsilon) \cdot \beta = \beta^{-1} \cdot (A \cup \epsilon) \cdot \beta$.
Proposition~\ref{prop: primitive partition is normal} then implies the claim.
\end{proof}

We use the following to show that certain edge sets are not contained in a set in~$\cH$.

\begin{claim} \label{normal tree}
Let $X$ be a set of edges that contains an $\ep$-labeled spanning tree of $K_3^{\Gamma}$, and an edge labeled by $\alpha \in \Gamma - \{\epsilon\}$.
Suppose $X$ is contained in a set in $H \in \mathcal H$.
Then $H = E_{A \cup \epsilon}$, where $A$ is the set in $\mathcal A$ that contains $\alpha$. 
\end{claim}
\begin{proof}
If $H$ contains a path,
then $H$ also contains the balanced cycle obtained by completing that path to a cycle.
Therefore we may assume that $X$ is closed under completing paths to balanced cycles.
Therefore $E_{\epsilon} \subseteq  X$.
By symmetry we may assume that $\alpha_{12} \in X$.
Since $X$ is closed under completing paths to balanced cycles, either $\alpha_{ij}$ or $\alpha^{-1}_{ij}$ is in $X$ for all $i,j \in \{1,2,3\}$.
Let $A \in \cA$ so that $\alpha \in A$.

Say $H$ is obtained from $E_{A' \cup \epsilon}$ for some $A' \in \cA$ by switching at each vertex $i$ with some $a_i \in \Gamma$.
Since $H$ contains $X$ which contains $\epsilon_{12}$, we have that $a_1^{-1} \cdot (A' \cup \epsilon) \cdot a_2$ contains $\epsilon$.
Claim~\ref{group theory} implies that $a_1^{-1} \cdot (A' \cup \epsilon) \cdot a_2 = A_{12} \cup \epsilon$ for some $A_{12} \in \cA$.
So the labels on the edges of $H$ between $1$ and $2$ are the elements of $A_{12} \cup \epsilon$.
Similarly, we can argue that between $i$ and $j$,
the edge labels are the elements of $A_{ij} \cup \epsilon$ for some $A_{ij} \in \cA$.
But since $\alpha_{ij}$ or $\alpha^{-1}_{ij}$ is in $X$ and $X \subseteq H$ it follows that $A_{ij}$ contains $\alpha_{ij}$ or $\alpha^{-1}_{ij}$, so $A_{ij} = A$ for all $i,j \in \{1,2,3\}$, and $H = E_{A \cup \epsilon}$.
\end{proof}

A set of edges of $K_3^{\Gamma}$ is \emph{balanced} if it contains no unbalanced cycles.

\begin{claim}\label{sets contained in hyperplanes}
If a set $X$ of edges has at most $3$ edges or consists of a balanced cycle with an extra edge, then $X$ is contained in some $H \in \cH$.
\end{claim}
\begin{proof}
First suppose that $X$ has at most $3$ edges.
If $X$ does not contain a spanning tree of $K_3^{\Gamma}$ then $X$ is contained in a set in $\mathcal H$ of type (b).
So assume otherwise.
If $X$ is balanced then $X$ is switching equivalent to a subset of $E_\epsilon$ and thus lies in $E_{A \cup \ep}$ for every $A \in \cA$.
If $X$ is not balanced, then $X$ is switching equivalent to the graph obtained by adding
a single edge of the form $\alpha_{ij}$ to a spanning tree where every edge has label $\epsilon$.
Let $A \in \mathcal{A}$ contain $\alpha$. 
Then $X$ is contained in a set switching equivalent to $E_{A \cup \epsilon}$.

It remains to consider the case that $X$ consists of a balanced triangle with a doubled edge.
By switching, we may assume that each edge of the triangle is $\epsilon$.
Without loss of generality assume the extra edge is $\alpha_{12}$.
Then $X$ is contained in $E_{A \cup \epsilon}$ where $A \in \mathcal{A}$ contains $\alpha_{12}$.
\end{proof}

We now show that $\cH$ is the collection of hyperplanes of a matroid $M$.
We must show that for all distinct $H_1, H_2 \in \cH$ and $e \notin H_1 \cup H_2$, there is a set in $\cH$ that contains $H_1 \cap H_2$ and $e$.
This is clearly true if $H_1$ and $H_2$ are both type (b).
If $H_1$ is type (a) and $H_2$ is type (b) then, up to switching equivalence and relabeling vertices, $H_1 \cap H_2$ consists of all edges between vertices $1$ and $2$ with label in $A \cup \epsilon$ for some $A \in \mathcal A$.
Then $e$ has vertex $3$ as an end, and by switching at vertex $3$ we may assume that $e$ is labeled by $\epsilon$.
Then $(H_1 \cap H_2) \cup e$ is contained in $E_{A \cup \ep}$, which is in $\cH$.
So we assume that $H_1$ and $H_2$ are both type (a).
For $i \in \{1,2\}$, say $H_i$ is switching equivalent to $E_{A_i \cup \epsilon}$ for $A_i \in \cA$.
We may assume that $H_1 \cap H_2$ contains a spanning tree.
Otherwise, if $e$ has the same ends as each edge in $H_1 \cap H_2$, then $(H_1 \cap H_2) \cup e$ is contained in a set of type (b).
If $e$ has an end not incident to any edge in $H_1 \cap H_2$, then we can argue as in the case where $H_1$ is of type (a) and $H_2$ of type (b) that $(H_1 \cap H_2) \cup e$ is contained in a set of type (a).
By switching equivalence, we may further assume that the edges of a spanning tree of $H_1 \cap H_2$ are all labeled by $\epsilon$.
If $H_1 \cap H_2$ contains an edge labeled by $\alpha \in \Gamma - \{\epsilon\}$, then $H_1 = H_2$ by Claim \ref{normal tree}, a contradiction.
So $H_1 \cap H_2$ is a balanced cycle, and therefore $(H_1 \cap H_2) \cup e$ is contained in a set in $\cH$ by Claim \ref{sets contained in hyperplanes}.
Thus, $\cH$ is the hyperplane set of a matroid $M$.

Next, $M$ is in fact a rank-$2$ lift of the graphic matroid $M(K_3^{\Gamma})$.
The only nonempty proper flats of $M(K_3^{\Gamma})$ are the parallel classes of edges.
All such sets are flats of $M$ as well, so $M$ is a lift of $M(K_3^{\Gamma})$ by \cite[Prop. 7.3.6]{oxley2006matroid}.
$M$ is not an elementary lift because every set of three edges is contained in a hyperplane by Claim~\ref{sets contained in hyperplanes}.
Given $\alpha,\beta \in \Gamma -  \{\epsilon\}$ such that no $A \in \mathcal{A}$ contains both $\alpha$ and $\beta$,
the set  $\{\epsilon_{12}, \alpha_{12}, \epsilon_{23}, \beta_{23}\}$ is not contained in a set in $\cH$ by Claim \ref{normal tree}.
Therefore the rank of $M$ is $4$, i.e. $M$ is a rank-$2$ lift of $M(K_3^\Gamma)$.

Finally, we show that a cycle of $K_3^{\Gamma}$ is a circuit of $M$ if and only if it is balanced.
Claim~\ref{sets contained in hyperplanes} implies that no balanced cycle is contained in a basis.
Now suppose $C$ is an unbalanced cycle.
Up to switching equivalence and relabeling vertices we may assume that $C = \{\epsilon_{12}, \epsilon_{23}, \alpha_{13}\}$.
Consider $B = C \cup \beta_{12}$ where $\beta$ and $\alpha$ are not in the same set in $\mathcal A$.
Then $B$ is not contained in a set in $\cH$ by Claim \ref{normal tree}.
Therefore $B$ is a basis of $M$ and thus $C$ is independent in $M$.
\end{proof}

We now prove Theorem~\ref{all groups converse}.

\begin{proof}[Proof of Theorem~\ref{all groups converse}]
    Let $M \in \mathcal{M}_{n,\Gamma}$ and assume $M$ has rank at least $n+1$, i.e.~is a non-elementary lift of $M(K_n^\Gamma)$.
    It follows from~\cite[Lemma~5.3]{walsh2022new} that $\Gamma$ has a nontrivial partition.
    Theorem~\ref{thm: 4 or more} implies that if $n \ge 4$ then $\Gamma = \mathbb{Z}_p^j$ for a prime $p$ and integer $j \ge 2$.
    
    Conversely, if $\Gamma$ has a nontrivial partition and $n=3$ then Theorem~\ref{thm: partition implies higher rank lifts}
    implies that a desired lift exists.
    If $n \ge 4$, then the desired lift is given by Theorem~\ref{groups construction}.
\end{proof}

We comment that while Theorem \ref{thm: partition implies higher rank lifts} constructs a rank-$2$ lift of $M(K_3^{\Gamma})$ that respects balanced cycles, no such lift can be constructed using Theorem \ref{thm: new lift construction} and a matroid $N$ on the cycles of $K_3^{\Gamma}$ when $\Gamma$ is non-abelian; this can be proved in a similar manner to Theorem~\ref{thm: 4 or more}.
So the matroids constructed in the proof of Theorem \ref{thm: partition implies higher rank lifts} provide another family of examples for which Theorem \ref{thm: new lift construction} does not apply.
In particular, if $\Gamma$ is non-abelian and $M$ is a $2$-element extension of one of the matroids from Theorem \ref{thm: partition implies higher rank lifts} by a set $X$ so that $M/X = M(K_3^{\Gamma})$, then $M$ non-representable by Proposition \ref{prop: representable lifts}.
Finally, we point out that Theorem \ref{thm: partition implies higher rank lifts} constructs a rank-$2$ lift of $M(K_3^{\Gamma})$, while for certain abelian groups it is possible to construct higher-rank lifts, as shown by Theorem \ref{groups construction}.
Is this possible for non-abelian groups as well? 
We expect an affirmative answer when $\Gamma$ has no two-element generating set.

\bibliographystyle{plain}
\bibliography{references.bib}

\begin{thebibliography}{10}

\bibitem{baer1961partitionen}
R.~Baer.
\newblock Partitionen endlicher gruppen.
\newblock {\em Mathematische Zeitschrift}, 75(1):333--372, 1961.

\bibitem{brylawski1986}
T.~H. Brylawski.
\newblock Constructions.
\newblock In Neil White, editor, {\em Theory of matroids}. Cambridge University
  Press, Cambridge, 1986.

\bibitem{Crapo}
H.~H. Crapo.
\newblock Single-element extensions of matroids.
\newblock {\em J. Res. Nat. Bur. Standards Sect. B}, 61B:55--65, 1965.

\bibitem{group-labelings}
M.~DeVos, D.~Funk, and I.~Pivotto.
\newblock When does a biased graph come from a group labeling?
\newblock {\em Adv. in Appl. Math.}, 61:1--18, 2014.

\bibitem{ingleton1971}
A.~W. Ingleton.
\newblock Representations of matroids.
\newblock In D.~J.~A. Welsh, editor, {\em Combinatorial mathematics and its
  applications (Proceedings of a conference held at the Mathematical Institute,
  Oxford, from 7-10 July, 1969)}. Academic Press, 1971.

\bibitem{IngletonMain1975}
A.~W. Ingleton and R.~A. Main.
\newblock Non-algebraic matroids exist.
\newblock {\em Bull. London Math. Soc.}, 7:144--146, 1975.

\bibitem{Longyear}
J.~Longyear.
\newblock The circuit basis in binary matroids.
\newblock {\em J. Number Theory}, 12:71--76, 1980.

\bibitem{PeterJorn}
P.~Nelson and J.~van~der Pol.
\newblock Doubly exponentially many {I}ngleton matroids.
\newblock {\em SIAM J. Disc. Math.}, 32(2):1145--1153, 2018.

\bibitem{oxley2006matroid}
J.~Oxley.
\newblock {\em Matroid theory}, volume~2.
\newblock Oxford University Press, USA, 2006.

\bibitem{OxleyWang}
J.~Oxley and S.~Wang.
\newblock Dependencies among dependencies in matroids.
\newblock {\em Electron. J. Combin.}, 26, 2019.

\bibitem{Jurrius}
F.~Ragnar, R.~Jurrius, and O.~Kuznetsova.
\newblock Combinatorial derived matroid.
\newblock {\em arXiv preprint arXiv:2206.06881}, 2022.

\bibitem{walsh2022new}
Z.~Walsh.
\newblock A new matroid lift construction and an application to group-labeled
  graphs.
\newblock {\em Electr. J. Combin.}, 29, 2022.

\bibitem{young1927partitions}
J.~W. Young.
\newblock On the partitions of a group and the resulting classification.
\newblock {\em Bulletin of the AMS}, 33(4):453--461, 1927.

\bibitem{zappa2003partitions}
G.~Zappa.
\newblock Partitions and other coverings of finite groups.
\newblock {\em Illinois J. of Math.}, 47(1-2):571--580, 2003.

\bibitem{zaslavsky1991liftmatroids}
T.~Zaslavsky.
\newblock Biased graphs. {II}. {T}he three matroids.
\newblock {\em J. Combin. Theory Ser. B}, 51:46--72, 1991.

\end{thebibliography}

\end{document}